\documentclass[letterpaper, 10pt, journal, twocolumn, final]{IEEEtran}

\usepackage[latin1]{inputenc}

\usepackage{amssymb,amsmath,amsfonts,amsthm}
\usepackage{microtype}
\usepackage{url}
\usepackage[noadjust]{cite}
\usepackage{hyperref}
\usepackage{xifthen}
\usepackage{xpatch}

\usepackage[dvipsnames]{xcolor}
\usepackage{tikz}
\usepackage{subfig}
\usepackage{graphicx}

\usepackage{algorithm}
\usepackage{algpseudocode}
\usepackage{algorithmicx}

\newtheorem{theorem}{Theorem}

\newtheorem{corollary}{Corollary}
\newtheorem{lemma}{Lemma}
\newtheorem{assumption}{Assumption}
\newtheorem{remark}{Remark}

\newcommand{\R}{\mathbb{R}}

\newcommand{\set}{\{1,\dots,N\}}
\newcommand{\col}{\textsc{col}}
\newcommand{\Xtp}{X_{t+1}}
\newcommand{\Xt}{X_{t}}

\newcommand{\Xit}{X_{i,t}}
\newcommand{\xtp}{x_{t+1}}
\newcommand{\ytp}{y_{t+1}}
\newcommand{\stp}{s_{t+1}}
\newcommand{\xt}{x_{t}}
\newcommand{\txt}{\tilde{x}_{t}}
\newcommand{\yt}{y_{t}}
\newcommand{\st}{s_{t}}
\newcommand{\xtpstar}{x_{t+1}^\star}

\newcommand{\xtstar}{x_t^\star}

\newcommand{\xitp}{x_{i,t+1}}
\newcommand{\sitp}{s_{i,t+1}}
\newcommand{\yitp}{y_{i,t+1}}
\newcommand{\xit}{x_{i,t}}
\newcommand{\txik}{\tilde{x}_{i,t}}
\newcommand{\sit}{s_{i,t}}
\newcommand{\yit}{y_{i,t}}
\newcommand{\ytpavg}{\bar{y}_{t+1}}
\newcommand{\stpavg}{\bar{s}_{t+1}}
\newcommand{\ytavg}{\bar{y}_t}
\newcommand{\stavg}{\bar{s}_t}
\newcommand{\dk}{d_t}

\newcommand{\ftp}{f_{t+1}}
\newcommand{\ft}{f_{t}}
\newcommand{\fitp}{f_{i,t+1}}
\newcommand{\fit}{f_{i,t}}

\newcommand{\phitp}{\phi_{t+1}}
\newcommand{\phiitp}{\phi_{i,t+1}}
\newcommand{\phit}{\phi_{t}}
\newcommand{\phiit}{\phi_{i,t}}

\newcommand{\1}{\mathbf{1}}

\newcommand{\norm}[1]{\left \|#1 \right \|}

\newcommand{\I}{\tfrac{\1\1^\top}{N}}
\newcommand{\blkdiag}{\text{blkdiag}}
\newcommand{\T}{^\top}
\newcommand{\vertiii}[1]{{\left\vert\kern-0.25ex\left\vert\kern-0.25ex\left\vert #1 
		\right\vert\kern-0.25ex\right\vert\kern-0.25ex\right\vert}}
\newcommand{\ped}{\theta}
\newcommand{\dist}{\mathrm{dist}}

\newcommand{\rhomax}{\rho_{\text{max}}}

\newcommand\oprocendsymbol{\hbox{$\square$}}
\newcommand\oprocend{\relax\ifmmode\else\unskip\hfill\fi\oprocendsymbol}
\def\eqoprocend{\tag*{$\square$}}

\graphicspath{{figs/}}

\def\algo/{Projected Aggregative Tracking}

\begin{document}

\title{Distributed Online Aggregative Optimization\\ 
  for Dynamic Multi-robot Coordination}

\author{Guido Carnevale, 
  Andrea Camisa, Giuseppe Notarstefano
  \thanks{G.~Carnevale, A.~Camisa and G.~Notarstefano are with the Department of Electrical, 
  Electronic and Information Engineering, University of Bologna, Bologna, Italy,
  \texttt{\{guido.carnevale, a.camisa, giuseppe.notarstefano\}@unibo.it}.
  This result is part of a project that has received funding from the European 
  Research Council (ERC) under the European Union's Horizon 2020 research 
  and innovation programme (grant agreement No 638992 - OPT4SMART).
  }
}

\maketitle

\begin{abstract}
  This paper focuses on an online version of the emerging distributed
  constrained aggregative optimization framework, which is particularly suited
  for applications arising in cooperative robotics. Agents in a network want to
  minimize the sum of local cost functions, each one depending both on a local
  optimization variable, subject to a local constraint, and on an aggregated
  version of all the variables (e.g., the mean). We focus on a challenging
  online scenario in which the cost, the aggregation functions and the
  constraints can all change over time, thus enlarging the class of captured
  applications. Inspired by an existing scheme, we propose a distributed
  algorithm with constant step size, named Projected Aggregative Tracking, to
  solve the online optimization problem. We prove that the dynamic regret is
  bounded by a constant term and a term related to time variations. Moreover, in
  the static case (i.e., with constant cost and constraints), the solution
  estimates are proved to converge with a linear rate to the optimal
  solution. Finally, numerical examples show the efficacy of the proposed
  approach on a robotic surveillance scenario.
\end{abstract}

\section{Introduction}
\label{sec:introduction}

Distributed optimization captures a variety of estimation and learning problems
over networks, including distributed data classification and localization in
smart sensor networks, to name a few.  The term ``online'' refers to scenarios
in which the problem data is not available a-priori, but rather it arrives
dynamically while the optimization process is executed.  In this paper, we
consider an online distributed optimization set-up in which agents in a network
must cooperatively minimize the sum of local cost functions that depend both on
a local optimization variable and on a global variable obtained by performing
some kind of aggregation of all the local variables (as, e.g., the mean).  This
\emph{aggregative optimization} set-up was introduced in the pioneering
work~\cite{li2020distributed}.
The framework is fairly general and embraces several control applications
of interest such as cooperative robotics and multi-vehicle surveillance.
Originally, it stems from distributed aggregative games~\cite{koshal2016distributed,
liang2017distributed,gadjov2018passivity,yi2019operator,belgioioso2020distributed},
where however the objective is to compute a (generalized) Nash equilibrium
rather than an optimal solution cooperatively.

There exists a vast literature on distributed online optimization, which
addresses two main optimization set-ups known as cost-coupled (or consensus
optimization) and constraint-coupled, see,
e.g.,~\cite{notarstefano2019distributed}. In the cost-coupled framework the goal
is to optimize a cost function given by the sum of several local functions with
a common decision variable.
Distributed online algorithms based on subgradient schemes are proposed in
\cite{cavalcante2013distributed,towfic2014adaptive,akbari2015distributed}, while
primal-dual or dual approaches are proposed in
\cite{mateos2014distributed,hosseini2016online,yuan2017adaptive}. The use of a
model for the minimum variation is considered
in~\cite{shahrampour2017distributed}, where a mirror descent algorithm is proposed.
Distributed online optimization is used in~\cite{zhou2017incentive} to
handle the distribution grids problem. In~\cite{akbari2019individual} an online
algorithm based on the alternating direction method of multipliers is proposed.
In the more recent constraint-coupled optimization framework, each local
objective function depends on a local decision variable, but all the variables
are coupled through separable coupling constraints. In~\cite{lee2017sublinear}
this set-up is addressed in an online setting and a sublinear regret bound is
ensured by using a distributed primal dual algorithm. The same result is
achieved in~\cite{li2020distributedConstraints} by introducing a push-sum
mechanism to allow for directed graph topologies.
Time-varying inequality constraints have
been taken into account in~\cite{yi2020distributed}, in which a distributed primal dual
mirror descent algorithm is proposed.

We recall that the above works are suited for cost-coupled %
and constraint-coupled set-ups.
The aggregative optimization framework addressed in this paper has been
recently introduced in works~\cite{li2020distributed,li2020distributedOnline} which consider respectively a static
unconstrained framework and an online constrained one. The distributed
algorithms proposed in these two papers leverage a tracking action to reconstruct
both the aggregative variable and the gradient of the whole cost function.
This ``tracking action'' is based on dynamic average
consensus (see~\cite{zhu2010discrete,kia2019tutorial}) and has been introduced
in the gradient tracking scheme for cost-coupled
optimization~\cite{shi2015extra,varagnolo2015newton,dilorenzo2016next,nedic2017achieving,qu_harnessing_2018,xu2017convergence,xi2017add,xin2018linear,scutari2019distributed}.
The gradient tracking has been applied to online optimization
in~\cite{zhang2019distributed}, in~\cite{notarnicola2020personalized} where partially
unknown cost functions are considered, and in~\cite{carnevale2020distributed}
where adaptive momenta are used.
The contributions of this paper are as follows.
We focus on an online constrained aggregative optimization set-up over
peer-to-peer networks of agents inspired to the set-up studied in the seminal work~\cite{li2020distributedOnline}. Despite the fact that we do not assume boundedness of the gradients and the feasible sets, we demonstrate stronger theoretical results than the state of art, i.e., tighter regret bounds and, by using our method in a time-invariant problem, linear convergence rate (instead of a sublinear one) to the optimal solution.
Additionally, the optimization set-up we consider enlarges the scope of previous works.
Indeed, we consider a wider time-varying framework in which also the feasible sets and the aggregation rules vary over time. These generalizations introduce additional terms in the regret analysis and thus pose new challenges that must be appropriately handled.
In \algo/, each
agent projects the updated local solution estimate on its time-varying
constraint set and then performs a convex combination with the current
estimate. Moreover, the trackers of the aggregative variable are generalized to
handle time-variation of the aggregation rules.
Under mild assumptions on the cost and constraint variations, we
provide a bound about the dynamic regret for the proposed scheme. We additionally provide a regret result on the violation of the
time-varying constraints.
In order to obtain
this result we study a dynamical system describing the algorithmic evolution
of: (i) the error of the solution estimate with respect to the minimum, (ii)
the consensus error of the aggregative variable trackers, and (iii) the
consensus error of the global gradient trackers.  Such dynamics is characterized
by a Schur system matrix, by which we are then able to
draw conclusions on the dynamic regret.
For the static case
(with constant cost and constraints), we show that the algorithm iterates converge to
the optimal solution with linear rate. Notably, our algorithm allows for
a constant step-size.
To corroborate the theoretical analysis, we show numerical simulations from a
cooperative robotics scenario in which robots have to accomplish a surveillance task.
The proposed scenario is dynamic and intrinsically characterized by time-varying
cost functions, aggregation functions and constraints, thus it cannot
be addressed by using state-of-art techniques.

The rest of the paper is organized as follows. Section~\ref{sec:problem_formaulation} describes the distributed online aggregative optimization framework and presents the distributed algorithm with its convergence properties. The algorithm analysis is performed
in Section~\ref{sec:analysis}. Finally, Section~\ref{sec:numerical_experiments} shows the effectiveness of our method.

\emph{Notation:}
We use $\col (v_1, \ldots, v_n)$ to denote the vertical concatenation of the column
vectors $v_1, \ldots, v_n$. We use $\blkdiag(M_1,\dots,M_N)$ to denote the block diagonal matrix where the $i$-th diagonal block is given by the matrix $M_i \in \R^{n_i\times m_i}$ for all $i \in \{1,\dots,N\}$.
The Kronecker product is denoted by
$\otimes$. The identity
matrix in $\R^{m\times m}$ is $I_m$, while $0_m$ is the zero matrix in
$\R^{m\times m}$. The column vector of $N$ ones is denoted by $1_N$ and 
we define $\1 \triangleq 1_N \otimes I_d$.
Dimensions are omitted whenever they are clear from the  context. Given a closed and convex set $X$, we use $P_X[y]$ to denote the projection of a vector $y$ on a $X$, namely $P_X[y] = \arg \min_{x \in X}\norm{x-y}$, while we use $\dist(y,X)$ to denote its distance from the set, namely $\dist(y,X) = \min_{x \in X}\norm{x-y}$. Given $x \in \R^n$, we use $[x]^+$ to denote $\max\{0,x\}$ in a component-wise sense. 
Let $M\in \R^{n\times n}$, then we denote as $\rhomax(M)$ its spectral radius.

\section{Problem Formulation and Algorithm Description}
\label{sec:problem_formaulation}
In this paper, we consider distributed online aggregative optimization problems
that can be written as
\begin{align}
\label{eq:online_aggregative_problem}
\begin{split}
	\min_{(x_1,\dots,x_N) \in \Xt} \: & \: \sum_{i=1}^{N}\fit(x_i,\sigma_t(x))
\end{split}
\end{align}
in which $x := \col(x_1, \dots, x_N) \in \R^n$ is the global decision vector,
with each $x_i \in \R^{n_i}$ and $n = \sum_{i=1}^{N} n_i$. The global decision
vector at time $t$ is constrained to belong to a set $\Xt \subseteq \R^n$ that can
be written as $\Xt = (X_{1,t} \times \ldots \times X_{N,t})$, where each
$\Xit \subseteq \R^{n_i}$. The functions $\fit: \R^{n_i} \times \R^d \to \R$
represent the local objective functions at time $t$, %
while the \emph{aggregation function} $\sigma_t(x)$ has the form
\begin{align}
	\sigma_t(x) := \frac{\sum_{i=1}^{N}\phiit(x_i)}{N},
\end{align}
where each $\phiit : \R^{n_i} \to \R^d$ is the $i$-th contribution to the aggregative variable at time $t$.
We compactly denote the cost function of problem~\eqref{eq:online_aggregative_problem} as $\ft(x,\sigma_t(x)) := \sum_{i=1}^N \fit(x_i,\sigma_t(x))$.
In problem \eqref{eq:online_aggregative_problem}, $\ft(\cdot,\sigma_t(\cdot))$
is not known to any agent: each of them can only privately access $\fit$, $\Xit$, and $\phiit$. We remark that each
agent $i$ accesses its private information $\fit$, and $\phiit$ only once its estimate $\xit$ has been computed. The idea is to solve problem~\eqref{eq:online_aggregative_problem} in a distributed way over a network of $N$
agents communicating according to a graph
$\mathcal{G} := (\set,\mathcal{E}, \mathcal{A})$, where $\set$ is the set
of agents, $\mathcal{E} \in \set \times \set$ is the
edges set, and $\mathcal{A} \in \R^{N \times N}$ is the weighted adjacency
matrix. Each agent $i$ can exchange data only with its
neighbors defined by $\mathcal{E}$.

The goal is to design distributed algorithms to seek a minimum for problem~\eqref{eq:online_aggregative_problem}. Next, we will denote as
$\nabla_1\fit(\cdot,\cdot)$ and as $\nabla_2\fit(\cdot,\cdot)$ the
gradient of $\fit$ with respect to respectively the first argument and the second
argument. Moreover, we also introduce $G_t: \R^n \times \R^{Nd} \to \R^n$ defined as $G_t(x, s) := \nabla_1\ft(x, s) + \nabla\phi(x)\frac{\1}{N}\sum_{i=1}^N \fit(x_i,s_i)$, where $x := \col(x_1,\dots,x_N) \in \R^n$, $s := \col(s_1, \ldots, s_N) \in \R^{Nd}$ with each $x_i \in \R^{n_i}$, $s_i \in \R^d$ for all $i \in \{1, \ldots, N\}$, $\nabla_1 \ft(x,s) := \col(\nabla_1 f_{1,t}(x_{1},s_{1}),\dots,\nabla_1 f_{N,t}(x_{N},s_{N}))$, and $\nabla\phi(x) := \blkdiag(\nabla\phi_1(x_1),\dots,\nabla\phi_N(x_N)) \in \R^{n \times Nd}$.

Let $\xit$ be the solution estimate of the problem at time $t$
maintained by agent $i$, and let $\xtstar$ be the (unique) minimizer of
$\ft(x, \sigma_t(x))$ over the set $\Xt$. Indeed, as we will formalize within Assumption~\ref{ass:convexity}, strong convexity of $\ft(x,\sigma_t(x))$ guarantees existence (and uniqueness) of $\xtstar$.
Then, given a finite value $T > 1$, the agents want to minimize the dynamic regret:
\begin{align}\label{eq:regret}
R_T := \sum_{t=1}^T \ft(\xt,\sigma_t(\xt)) -  \sum_{t=1}^T\ft(\xtstar,\sigma_t(\xtstar)).
\end{align}  
Another popular metric is the so-called static regret~\cite{hosseini2016online}.
However, as done in most of the literature, we focus on~\eqref{eq:regret}, which is more challenging to handle.
To this end,
we propose our \algo/ algorithm.
Each agent $i$ maintains for each time instant $t$ an estimate $\xit$ of the
component $i$ of a minimum $\xtstar$ of problem~\eqref{eq:online_aggregative_problem}. In order to reconstruct the descent direction and use it to update the estimate $\xit$, agent $i$ needs to reconstruct the global information $\sum_{i=1}^{N}\frac{\phiit(\xit)}{N}$ and
$\sum_{i=1}^{N}\nabla_2
\fit\left(\xit,\sum_{j=1}^{N}\frac{\phi_{j,t}(x_{j,t})}{N}\right)$, which are not locally available. To overcome this lack of information, agent $i$ maintains auxiliary variables $\sit$ and $\yit$ and iteratively updates them according to a perturbed consensus mechanism.
A pseudo-code of the \algo/ algorithm is reported in Algorithm~\ref{table:constrained_aggregative_GT}
from the perspective of agent $i$, in which $\alpha$ is a positive constant
step-size, $\delta \in (0,1)$ is a constant algorithm parameter, and each
element $a_{ij}$ represents the $(i,j)$ entry of the weighted adjacency
matrix $\mathcal{A}$ of the network.
\begin{algorithm}%
	\begin{algorithmic}
		\State initialization:
		\begin{align*}
		  x_{i,0} \in X_{i,0}, \:\:\:\: s_{i,0} = \phi_{i,0}(x_{i,0}), \:\:\:\: y_{i,0} = \nabla_2f_{i,0}(x_{i,0},s_{i,0})
		\end{align*}
		\For{$t=0, 1, \dots$}
		\begin{align*}
		\txik & = P_{\Xit} \big[\xit - \alpha(\nabla_1 \fit(\xit,\sit) + \nabla\phiit(\xit)\yit)\big]%
		\\
		\xitp & = \xit + \delta(\txik - \xit)%
		\\
		\sitp & = \sum_{j=1}^{N}a_{ij}s_{j,t} + \phiitp(\xitp) - \phiit(\xit)
		\\
		\yitp & = \sum_{j=1}^{N}a_{ij}y_{j,t} + \nabla_2 \fitp(\xitp, \sitp)
		\\
		&\hspace{1.84cm}
		- \nabla_2 \fit(\xit,\sit)%
		\end{align*}
		\EndFor
	\end{algorithmic}
	\caption{\algo/ (Agent $i$)}
	\label{table:constrained_aggregative_GT}
\end{algorithm}

\section{Convergence Analysis}
\label{sec:analysis}

This section gives the convergence properties of the proposed distributed algorithm.

\subsection{\algo/ Reformulation and Assumptions}
First of all let us rewrite all the agents' updates, according to
Algorithm~\ref{table:constrained_aggregative_GT}, in a stacked vector form as
\begin{subequations}
\begin{align}
\txt &= P_{\Xt}[\xt - \alpha(\nabla_1\ft(\xt,\st) + \nabla\phit(\xt)\yt)]\label{eq:tilde_gloabl_update}\\
\xtp &= \xt + \delta(\txt - \xt)\label{eq:x_global_update}\\
\stp &= A\st + \phitp(\xtp) - \phit(\xt)\label{eq:sigma_global_update}\\
\ytp &= A\yt + \nabla_2 \ftp(\xtp, \stp) - \nabla_2\ft(\xt,\st),\label{eq:y_global_update}
\end{align}%
\end{subequations}%
where we used the notation $\xt := \col(x_{1,t}, \dots, x_{N,t})$,
  $\st := \col(s_{1,t}, \dots, s_{N,t})$, and
  $\yt := \col(y_{1,t}, \dots, y_{N,t})$. Moreover, we also introduced the symbols
  $\nabla_2 \ft(\xt,\st) := \col(\nabla_2
  f_{1,t}(x_{1,t},s_{1,t}),\dots, \nabla_2 f_{N,t}(x_{N,t},s_{N,t}))$,
  $\nabla \phit(\xt) :=
  \blkdiag(\nabla\phi_{1,t}(x_{1,t}),\dots,\nabla\phi_{N,t}(x_{N,t}))$ and
  $A := \mathcal{A} \otimes I$.
In order to perform the convergence analysis, we derive bounds for the quantities
$\norm{\xtp - \xt}$, $\norm{\xtp - \xtpstar}$, $\norm{\ytp - \1\ytpavg}$, and
$\norm{\stp - \1\stpavg}$, in which $\ytavg := \frac{1}{N}\sum_{i=1}^N\yit$ and
$\stavg := \frac{1}{N}\sum_{i=1}^N\sit$ denote the mean vectors of $\yt$ and $\st$,
respectively. Let $z_t$ be the vector staking the above quantities %
\begin{align}
	z_t := \begin{bmatrix}
		\norm{\xt - \xtstar}
		\\
		\norm{\st - \1\stavg}
		\\
		\norm{\yt - \1\ytavg}
	\end{bmatrix}.\label{eq:zt}
\end{align}
	Moreover, also the following variables will be useful to provide the main result of the paper, namely
	\begin{subequations}\label{eq:bounds_variables}
	\begin{align}
		\eta_{t}&:= \sup_{x \in \R^n, z \in \R^{Nd}} \norm{\nabla_2\ftp(x,z) - \nabla_2 \ft(x,z)}\label{eq:eta}
		\\
		\omega_{t}&:= \sup_{x \in \R^n} \norm{\phitp(x) - \phit(x)}\label{eq:omega}
		\\
		\gamma_t &:= \sup_{x \in \R^{n}} \left|\dist(x,\Xtp) - \dist(x,\Xt)\right|\label{eq:gamma}
		\\
		\zeta_t &:=\|\xtpstar-\xtstar\|,\label{eq:zeta}
	\end{align}
	\end{subequations}	
	where we recall that $\xtstar$ is the optimal solution of $\ft(\xt,\sigma_t(\xt))$.
Next, we state the assumptions of our framework.

\begin{assumption}[Communication graph]\label{ass:network}
  The graph $\mathcal{G}$ is undirected and connected and $\mathcal{A}$ is doubly stochastic.
  \oprocend
\end{assumption}
\begin{assumption}[Convexity]\label{ass:convexity}
	For all $i \in \set$ and all $t \ge 0$, $\Xit \subseteq \R^{n_i}$ is nonempty, closed and convex, while the global objective function $\ft(x,\sigma_t(x))$ is $\mu$-strongly convex.	\oprocend
\end{assumption}
\begin{assumption}[Function Regularity]\label{ass:lipschitz}
For all $t \ge 0$, the function $\ft(x,\sigma_t(x))$ is differentiable with $L_1$-Lipschitz continuous gradients,
and %
$G_t(x,s)$, $\nabla_2 \ft(x,s)$ are Lipschitz continuous %
with constants $L_1, L_2 > 0$, respectively.
            For all $i \in \set$ and $t \ge 0$, the aggregation function $\phiit(x_i)$ is differentiable
        	and $L_3$-Lipschitz continuous, and $\eta_t$ and $\omega_t$ are finite.
	\oprocend
\end{assumption}
We start by noting that
\begin{subequations}
\begin{align}
\stpavg &= \stavg + \frac{\1\T}{N}(\phitp(\xtp) - \phit(\xt))\label{eq:sigma_mean_update}
\\
\ytpavg &= \ytavg \! + \! \frac{\1\T}{N}(\nabla_2\ftp(\xtp,\stp) \! - \! \nabla_2\ft(\xt,\st)).\label{eq:y_mean_update}
\end{align}
\end{subequations}
Then, if we initialize $\sigma$ and $y$ as $\sigma_0 := \phi_0(x_0)$ and
$y_0 := \nabla_2f_0(x_0,s_0)$, from~\eqref{eq:sigma_mean_update}
and~\eqref{eq:y_mean_update}, it holds for all $t \ge 0$
\begin{subequations}
\begin{align}
\stavg &= \frac{1}{N}\sum_{i=1}^{N}\phit(\xit) := \sigma_t(\xt)
\label{eq:s_mean}
\\
\ytavg &= \frac{1}{N}\sum_{i=1}^{N}\nabla_2\fit(\xit,\sit)\label{eq:y_mean}.
\end{align}
\end{subequations}

\subsection{Preparatory Lemmas}

Here we present four preparatory Lemmas that we need to prove Theorem~\ref{th:dyn_regret}. For brevity, we will use $\dk$ to denote the descent direction used within the update~\eqref{eq:tilde_gloabl_update}, i.e.,
\begin{align}
	\dk &:= \nabla_1 \ft(\xt,\st) + \nabla\phit(\xt)\yt.\label{eq:dk}
\end{align}
\begin{lemma}\label{lemma:xk_xstar}
	Let Assumptions~\ref{ass:network},~\ref{ass:convexity}, and~\ref{ass:lipschitz} hold. If $\alpha \leq \frac{1}{L_1}$, then %
	\begin{align*}
	\norm{\xtp - \xtpstar} &\leq(1-\delta\mu\alpha)\norm{\xt-\xtstar} + \delta\alpha L_1\norm{\st-\1\stavg}\\
	& \hspace{0.5cm}+ \delta\alpha L_3\norm{\yt - \1\ytavg} + \zeta_t.
	\end{align*}
\end{lemma}
\begin{proof}
The proof is provided in Appendix~\ref{sec:proof_xk_xkstar}.
\end{proof}
\begin{lemma}\label{lemma:xkp_xk}
	Let Assumptions~\ref{ass:network},~\ref{ass:convexity}, and~\ref{ass:lipschitz} hold. Then
	\begin{align*}
	\norm{\xtp - \xt} &\leq \delta(2 + \alpha L_1 + \alpha L_1L_3)\norm{\xt -\xtstar}\notag\\
	 &\hspace{0.5cm}+ \delta\alpha L_1\norm{\st-\1\stavg}+\delta\alpha L_3\norm{\yt - \1\ytavg}.
	\end{align*}
\end{lemma}
\begin{proof}
	The proof is provided in Appendix~\ref{sec:proof_xkp_xk}.
\end{proof}
\begin{lemma}\label{lemma:sk_skavg}
	Let Assumptions~\ref{ass:network},~\ref{ass:convexity}, and~\ref{ass:lipschitz} hold. Then
	\begin{align*}
	\norm{\stp-\1\stpavg} &\leq \Lambda\norm{\st - \1\stavg} + \delta\alpha L_1L_3\norm{\st-\1\stavg}\notag
	\\
	&\hspace{0.5cm}+\delta(2 L_3 + \alpha L_1L_3 + \alpha L_1L_3^2)\norm{\xt -\xtstar}\notag
	\\
	&\hspace{0.5cm}+\delta\alpha L_3^2\norm{\yt - \1\ytavg}  + \omega_t,
	\end{align*}
	where $\Lambda$ is the maximum eigenvalue of the matrix $A - \I$.
\end{lemma}
\begin{proof}
	The proof is provided in Appendix~\ref{sec:proof_sk_skavg}.
\end{proof}
\begin{lemma}\label{lemma:yk_ykavg}
	Let Assumptions~\ref{ass:network},~\ref{ass:convexity}, and~\ref{ass:lipschitz} hold. Then
	\begin{align}
	&\norm{\ytp - \1\ytpavg} \leq \Lambda\norm{\yt-\1\ytavg}\notag
	\\
	& \hspace{0.3cm}
	+ \delta\alpha L_3(L_2 + L_2L_3)\norm{\yt - \1\ytavg}\notag
	\\
	& \hspace{0.3cm}+ \delta(2 + \alpha L_1 + \alpha L_1)(L_2 + L_2L_3)\norm{\xt -\xtstar}\notag
	\\
	&\hspace{0.3cm}+ \delta\alpha L_1(L_2 + L_2L_3)\norm{\st-\1\stavg}\notag
	\\
	&\hspace{0.3cm}
	+ 2L_2\norm{\st-\1\stavg} + L_2\omega_t + \eta_t,\notag
	\end{align}
	where $\Lambda$ is the maximum eigenvalue of the matrix $A - \I$.
\end{lemma}
\begin{proof}
	The proof is provided in Appendix~\ref{sec:proof_yk_ykavg}.
\end{proof}

\subsection{Regret analysis and linear rate in static set-up}
Now, we state the main theoretical results of the paper. Next theorem provides a bound on the dynamic regret %
of the iterates generated by the \algo/ distributed algorithm in the general, online set-up~\eqref{eq:online_aggregative_problem}.
\begin{theorem}\label{th:dyn_regret}
	Consider \algo/ as given in Algorithm~\ref{table:constrained_aggregative_GT}.
	Let Assumptions~\ref{ass:network},~\ref{ass:convexity}, and~\ref{ass:lipschitz} hold. Then, there exists $\lambda, \bar{\delta} > 0$ and $\tilde{\rho} \in (0,1)$ so that, if $\alpha \leq \frac{1}{L_1}$ and $\delta \in (0,\bar{\delta})$, it holds 
	\begin{equation}
		R_T \le \frac{L_1\lambda^2}{2}\left(\frac{\norm{z_0}^2}{1 - \tilde{\rho}^2}  + 2\norm{z_0} U_T + Q_T\right),\label{eq:dyn_regret_bound}
	\end{equation}
	where $R_T$ is defined as in~\eqref{eq:regret} and
	\begin{subequations}\label{eq:U_T_Q_T}
	\begin{align}
	U_T &:= \sum_{t=1}^T\sum_{k=0}^{t-1}\tilde{\rho}^{t+k}\bigg(\norm{\zeta_{t-k-1}} + 2\norm{\eta_{t-k-1}} 
	\notag\\
	&\hspace{3.6cm}
	+ (1+L_2)\norm{\omega_{t-k-1}}\bigg),
	\\
	Q_T &:= \sum_{t=1}^T\sum_{k=0}^{t-1}\tilde{\rho}^{2k}(\zeta_{t-k-1}^2 + 2\eta_{t-k-1}^2 + (1 + L_2)\omega_{t-k-1}^2).
	\end{align}
\end{subequations}
Moreover, if $\gamma_t$ (cf.~\eqref{eq:gamma}) is finite for all $t \ge 0$, then the constraint violation is bounded by 
\begin{align}\label{eq:constraint_violation}
	\sum_{t=1}^T \dist(x_t,\Xt) &\leq \frac{1}{1-(1-\delta)^T}\dist(x_0,X_0) 
	\notag\\
	&\hspace{0.5cm}
	+ \sum_{t=1}^T \sum_{k=0}^{t-1}(1-\delta)^k \gamma_{t-k-1}.
\end{align}
\end{theorem}
\begin{proof}
The proof is provided in Appendix~\ref{sec:proof_theorem}.
\end{proof}
Operatively, in order to choose an appropriate value of the parameter $\delta$, it is necessary to first estimate the upper bound $\bar{\delta}$. As it emerges from the proof of Theorem~\ref{th:dyn_regret}, this can be done as follows: (i) compute a matrix $M(\delta)$ (cf.~\eqref{eq:z}), which depends on the various problem constants and on $\delta$, (ii) compute $\bar{\delta}$ as the maximum value of $\delta$ such that all the eigenvalues of $M(\delta)$ are strictly in the unit circle.
		We observe that Theorem~\ref{th:dyn_regret} improves the dynamic regret bound provided in~\cite{li2020distributedOnline}, which demonstrates a bound of the type $O(T) + O(\sqrt{T}V_T)$ (where $V_T$ is a term capturing variations of the problem). The authors also show that there exists a particular, constant step-size that allows to tighten the first term to $O(\sqrt{T})$. However, the choice of the the step-size requires a prior knowledge of $T$ and $V_T$. In both cases, we improve the first term, which is replaced by the constant $\frac{L_1\lambda^2}{2}\frac{\norm{z_0}^2}{1 - \tilde{\rho}^2}$, while in our terms $U_T$ and $Q_T$ (cf.~\eqref{eq:U_T_Q_T}) the variations of the problem are scaled by $\tilde{\rho}^t$, i.e., an exponentially decaying quantity since $\tilde{\rho} \in (0,1)$.
\begin{remark}[Average Regret]
	\label{rem:avg_regret}
	Let us consider the case in which the problem variations are bounded by a constant, i.e., suppose there exists $C > 0$ so that $\zeta_t, \eta_t, \omega_t \leq C$ for all $t \ge 0$. In this case, by using the definitions of $U_T$ and $Q_T$ (cf.~\eqref{eq:U_T_Q_T}) and recalling that $\tilde{\rho} \in (0,1)$, we can use the geometric series property to get
	\begin{align*}
		U_T &\leq \frac{(4+L_2)C}{1 - \tilde{\rho}^T},
		\hspace{0.5cm}\text{and}\hspace{0.5cm}
		Q_T &\leq \frac{(3+(1+L_2))C^2T}{1 -\tilde{\rho}^{2}}.
	\end{align*}
	In this case, the average regret approaches a constant value,
	\begin{align*}
		\lim_{T \to \infty} R_T/T = \frac{L_1\lambda^2(4+L_2)C^2}{2(1 - \tilde{\rho}^2)^2}.
		\eqoprocend
	\end{align*}
\end{remark}
\begin{remark}[Inequality constraints]
	Consider the case in which $\Xit$ can be expressed in terms of inequality constraints, namely
	\begin{align*}
		\Xit := \{x_i \in \R^{n_i} \mid h_{i,t}(x_i) \le 0_{m_i}\},
	\end{align*}
	with $h_{i,t}: \R^{n_i} \to \R^{m_i}$ for all $i \in \set$ and $t \ge 0$.
	In this case, in place of the distance function $\dist(x, \Xt)$, one can use
	$\|[h_t(x)]^+\|$ as a metric to characterize the constraint violation,
	where $h_t(x) := \col(h_{1,t}(x_1),\dots,h_{N,t}(x_N))$. By repeating
	similar arguments as in the proof of Theorem~\ref{th:dyn_regret},
	one obtains similarly that $\sum_{t=1}^T \|[h_t(x_t)]^+\| \leq
	\frac{1}{1-(1-\delta)^T}\|[h_t(x_0)]^+\| 
		+ \sum_{t=1}^T \sum_{k=0}^{t-1}(1-\delta)^k \gamma_{t-k-1}$.
	\oprocend
\end{remark}
In the following corollary, we assess that in the static case
the \algo/ distributed algorithm converges to the (fixed) optimal
solution $x^\star$ with a linear rate.\footnote{%
A sequence $\{x_t\}$ converges linearly to $\bar{x}$
if there exists a number $\eta \in (0,1)$ such that
$\frac{\| x^{t+1} - \bar{x} \|}{\| x^{t} - \bar{x} \|} \to \eta$ as $t \to \infty$.}
\begin{corollary}[Static set-up]\label{cor:static}
	Under the same assumptions of Theorem~\ref{th:dyn_regret}, if it holds $\ft = f$, $\phit = \phi$, and $\Xit = X_i$ for all $i \in \{1, \ldots, N\}$ and all $t \ge 0$, then there exists $\lambda, \bar{\delta} > 0$ and $\tilde{\rho} \in (0,1)$ so that, if $\alpha \leq \frac{1}{L_1}$ and $\delta \in (0,\bar{\delta})$, it holds
	  \begin{equation*}
		f(\xt,\sigma(\xt)) - f(x^\star,\sigma(x^\star)) \le \tilde{\rho}^{2t}\frac{L_1 \lambda^2}{2} \norm{z_0}^2.
	  \end{equation*}
  \end{corollary}
\begin{proof}
The proof is provided in Appendix~\ref{sec:proof_cor_stat}.
\end{proof}
\section{Numerical experiments}%
\label{sec:numerical_experiments}

In this section we show the effectiveness of \algo/ on a multi-robot surveillance scenario.

\subsubsection{Online set-up}
\label{subsec:dynamic_case}
Let us consider a network of cooperating robots that aim to protect a
target with location $b_t \in \R^2$ at time $t$ from some intruders.
The optimization variables $\xit \in \R^2$
represent the position of robots at each time $t$ and each robot $i$ is able to move from
$\xit$ to $\xitp$ using a local controller.
We associate to each robot $i$ an intruder located at $p_{i,t} \in \R^2$
at time $t$. The dynamic protection strategy applied by each robot consists
of staying simultaneously close to the protected target and to the associated
intruder. Meanwhile, the whole team of robots tries to keep its
weighted center of mass rotating close to the target. A concept of this scenario is given in Fig.~\eqref{fig:concept}.
	\begin{figure}
		\centering
	\includegraphics[scale=1]{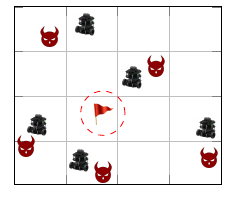}
	\caption{Multi- robot surveillance scenario - Robot icons denote agents, devil icons denote intruders, while the flag is the target to be protected.}
	\label{fig:concept}
	\end{figure}
This strategy is obtained by solving
problem~\eqref{eq:online_aggregative_problem} with the cost functions
$\fit(x_i,\sigma_t(x)) = \frac{1}{2}\norm{x_i - p_{i,t}}^2+ \frac{\gamma_1}{2}\norm{x_i - b_{t}}^2 + \frac{\gamma_2}{2N}\norm{\sigma_t(x) - b_t}^2$,
with $\gamma_1 = 1$, $\gamma_2 = 10$ and the aggregation rules $\phiit(x_i) = \beta_i x_i + a_{t}$, where $\beta_i > 0$ and $a_{t} \in \R^2$ represents a time-varying offset which follows the law $a_{t} = r\col(\cos(t/(2\pi\tau)),\sin(t/(2\pi\tau)))$ for some $r, \tau  > 0$. In this way, the center of mass $\frac{1}{N}\sum_{i=1}^N x_i^t$ is forced to rotate around the target position $b_t$.

We address a scenario with $N = 50$ agents and intruders.
As regards the constraints, we consider
a common time-varying box $\Xit = \{x \in \R^2 \mid 0 \leq x \leq u_{t} \}$ for all $i$, where $u_{t} \in \R^2$ starts from $[20,20]$ and linearly increases at each iteration. In this way, the agents initially stay closer to the target and then they move toward the associated intruders.
Each intruder $i$ moves along a circle of radius $r = 1$ according to the law $p_{i,t} = p_{i,c} + r\col(\cos(t/100),\sin(t/100))$, where $p_{i,c} \in \R^2$ is randomly generated. The target $b_{t}$ and the offset $a_{t}$ follow similar laws. In this setup, being the sinusoidal functions bounded, the constants $\eta_t$ and $\omega_t$ introduced in~\eqref{eq:bounds_variables} can be uniformly bounded as $\eta_t \leq \gamma_2\sqrt{N}r$ and $\omega_t \leq \sqrt{N}r$ for all $t \ge 0$. %
Moreover, the vector $u_t$ defining the box $\Xit$ changes linearly with respect time and, thus, also the constant $\gamma_t$ (cf.~\eqref{eq:gamma}) can be uniformly bounded.
As regards the algorithm parameters, we set $\alpha = 1$ and $\delta = 0.5$.
We performed $100$ Monte Carlo trials that differ in the problem parameters and agents' initial conditions. Fig.~\ref{fig:regret_avg} shows
that the behavior of the algorithm does not depend on the generated instances.
Indeed, the achieved average dynamic regret, as predicted in Remark~\ref{rem:avg_regret}, converges asymptotically to a constant.
\begin{figure}
	\centering
	\includegraphics[scale=0.75]{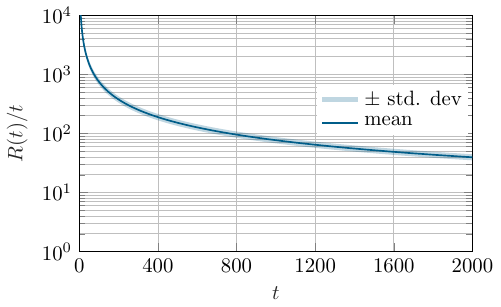}
	\caption{Online case -- Mean of the average Dynamic regret and $1$-standard deviation band
	over $100$ Monte Carlo trials.}
	\label{fig:regret_avg}
\end{figure}

\subsubsection{Static set-up}
Now we address a static instance of the problem. Namely, we
fix $\Xt$ and the positions of the intruders and of the target.
We perform a Monte Carlo simulation consisting of $100$ trials
on the same network of $N =50$ agents with the same algorithm parameters.
As predicted by Corollary~\ref{cor:static}, Fig.~\ref{fig:static}
shows an exponential decay of
$\frac{\norm{\xt - x^\star}}{\norm{x^\star}}$.%
\begin{figure}
	\centering
	\includegraphics[scale=0.75]{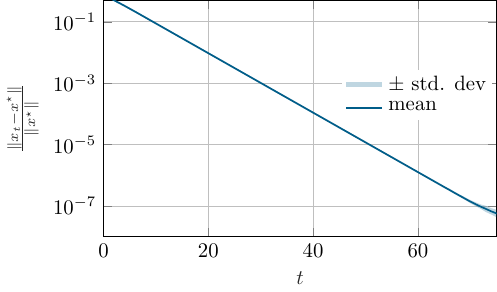}
	\caption{Static case --
	  Mean of the relative error and $1$-standard deviation band
	  obtained with $100$ Monte Carlo trials.}
	\label{fig:static}
\end{figure}

\section{Conclusions}	
\label{sec:conclusions}
In this paper, we focused on online instances of the distributed constrained aggregative optimization framework.
We proposed \algo/, a distributed algorithm allows for time-varying feasible sets and aggregation rules. We perform a regret analysis of the scheme by which we conclude that the dynamic regret is bounded by a constant term and a term related to time variations, while in the static case, the solution estimates linearly converge to the optimal solution.
Numerical computations confirmed our findings.
  
\begin{appendix}

\subsection{Proof of Lemma~\ref{lemma:xk_xstar}}
\label{sec:proof_xk_xkstar}
We begin by using~\eqref{eq:x_global_update}, which leads to
\begin{align}
&\norm{\xtp - \xtpstar} = \norm{\xt + \delta(\txt - \xt)- \xtpstar} \notag
\\
&\stackrel{(a)} \leq \norm{\xt + \delta(\txt - \xt)- \xtstar} +\norm{\xtpstar -\xtstar}\notag
\\
&\stackrel{(b)} \leq \norm{\xt + \delta(\txt - \xt)- \xtstar} + \zeta_t,\label{eq:before_fixed_point}
\end{align}
where in \emph{(a)} we add and subtract the term $\xtstar$ and use the triangle inequality,
and in \emph{(b)} we use $\zeta_t$ (cf~\eqref{eq:zeta}).
Being $\xtstar$ the minimizer of $\ft$ over $\Xt$,
then it holds %
$P_{\Xt}\left[\xtstar - \alpha\ft(\xtstar,\sigma_t(\xtstar))\right] = \xtstar$. %
Then, we add the null term $\delta\left(P_{\Xt}\left[\xtstar - \alpha\nabla\ft(\xtstar,\sigma_t(\xtstar))\right] - \xtstar\right)$
in the first norm of~\eqref{eq:before_fixed_point} and we apply the triangle inequality and~\eqref{eq:tilde_gloabl_update} to write
\begin{align}
&\norm{\xtp - \xtpstar}
\leq (1  -  \delta)\norm{\xt  - \xtstar}
\notag\\
&\hspace{0.5cm}
+  \delta\norm{P_{\Xt}[\xt  -  \alpha\dk]  -  P_{\Xt}\left[\xtstar  -  \alpha\nabla\ft(\xtstar,\sigma_t(\xtstar))\right]} 
+  \zeta_t
\notag\\
&\stackrel{(a)}{\leq}
(1  - \delta)\norm{\xt  - \xtstar} 
\notag\\
&\hspace{0.5cm}
+  \delta\norm{\xt  -  \alpha\dk  -  (\xtstar  -  \alpha\nabla\ft(\xtstar,\sigma_t(\xtstar)))}  +  \zeta_t,\label{eq:after_projection}
\end{align}
where \emph{(a)} uses the non-expansiveness of the projection, see~\cite{bertsekas2015convex}. Add and subtract within the second norm $\alpha\nabla\ft(\xt,\sigma_t(\xt))$ and apply the triangle inequality to rewrite~\eqref{eq:after_projection} as %
\begin{align}
&\norm{\xtp - \xtpstar} \leq(1-\delta)\norm{\xt-\xtstar} 
\notag\\
&
+ \delta\norm{\xt - \alpha\nabla\ft(\xt,\sigma_t(\xt)) - \left(\xtstar - \alpha\ft(\xtstar,\sigma_t(\xtstar))\right)} 
\notag\\
&\hspace{0.5cm}
+ \delta\alpha\norm{\dk - \nabla\ft(\xt,\sigma_t(\xt))} + \zeta_t
\notag\\
&\stackrel{(a)}\leq(1-\delta\mu\alpha)\norm{\xt-\xtstar} + \delta\alpha\norm{\dk - \nabla\ft(\xt,\sigma_t(\xt))} + \zeta_t,
\notag
\end{align} 
where \emph{(a)} uses~\cite[Lemma~3]{li2020distributed}. Add and subtract into the second norm $\nabla\phit(\xt)\1\frac{1}{N}\sum_{i=1}^{N}\nabla_2\fit(\xit,\sit)$, and rearrange as
\begin{align}
&\norm{\xtp - \xtpstar} \leq(1-\delta\mu\alpha)\norm{\xt-\xtstar}
\notag\\
&\hspace{0.5cm}
+ \delta\alpha\norm{G_t(\xt,\st) - \nabla f(\xt)}
\notag\\
&\hspace{0.5cm}+\delta\alpha\norm{\nabla\phit(\xt)\left(\yt -\1\frac{1}{N}\sum_{i=1}^{N}\nabla_2\fit(\xit,\sit)\right)} + \zeta_t
\notag\\
&\stackrel{(a)} = (1-\delta\mu\alpha)\norm{\xt-\xtstar}
+ \delta\alpha\norm{G_t(\xt,\st) - \nabla\ft(\xt,\sigma_t(\xt))}
\notag\\
&\hspace{0.5cm}+\delta\alpha\norm{\nabla\phit(\xt)\left(\yt -\1\ytavg\right)} + \zeta_t,\label{eq:xkp_xkpstar_before_lipschitz}
\end{align}
where in \emph{(a)} we use~\eqref{eq:y_mean}. %
Consider the term $\|G_t(\xt,\st) - \nabla\ft(\xt,\sigma_t(\xt))\|$. The definition of $G_t$ and~\eqref{eq:s_mean} give
\begin{align}
&\norm{G_t(\xt,\st) - \nabla\ft(\xt,\sigma_t(\xt))} 
\notag\\
&= \norm{G_t(\xt,\st) - \nabla \ft(\xt,\stavg)}
\stackrel{(a)}{\leq} 
L_1 \norm{\st - \1\stavg},\label{eq:dkavg_gk}
\end{align}
where \emph{(a)} uses the Lipschitz continuity of $G_t$ (cf. Assumption~\ref{ass:lipschitz}). The proof follows by~\eqref{eq:xkp_xkpstar_before_lipschitz},~\eqref{eq:dkavg_gk}, and $\norm{\nabla\phit(x)} \le L_3$ for all $x \in \R^n$ (which is derived from Assumption~\ref{ass:lipschitz}).	\oprocend

\subsection{Proof of Lemma~\ref{lemma:xkp_xk}}
\label{sec:proof_xkp_xk}
We can use~\eqref{eq:x_global_update} to write
\begin{align}
\norm{\xtp - \xt}  &= \norm{\xt + \delta(\txt - \xt) - \xt} 
\notag\\
&= \delta\norm{\txt - \xt}
\stackrel{(a)} =\delta\norm{P_{\Xt}[\xt - \alpha\dk] - \xt}\notag,
\end{align}
where in \emph{(a)} we have used the update~\eqref{eq:tilde_gloabl_update}. By adding the null quantity $\left(P_{\Xt}\left[\xtstar - \alpha\nabla\ft(\xtstar,\sigma_t(\xtstar))\right] - \xtstar\right)$ within the norm and applying the triangle inequality, we get
\begin{align}
&\norm{\xtp - \xt}\notag
\\
&\leq\delta\norm{P_{\Xt}[\xt - \alpha\dk] - P_{\Xt}\left[\xtstar - \alpha\nabla\ft(\xtstar,\sigma_t(\xtstar))\right]} 
\notag\\
&\hspace{0.5cm}
+ \delta\norm{\xt - \xtstar} \notag
\\
&\stackrel{(a)}\leq2\delta\norm{\xt -\xtstar} + \delta\alpha\norm{\dk - \nabla\ft(\xtstar,\sigma_t(\xtstar))},\notag
\end{align}
where in \emph{(a)} we use a projection property and the triangle inequality. We add and subtract within the norm the term $\nabla\phit(\xt)\1\sum_{i=1}^{N}\nabla_2\fit(\xit,\sit)$ and use the expression of $\dk$ and $G_t$ and the triangle inequality to write
\begin{align}
&\norm{\xtp - \xt} \leq 2\delta\norm{\xt -\xtstar}
\notag\\
&\hspace{0.5cm}
+ \delta\alpha\norm{G_t(\xt,\st) - \nabla\ft(\xtstar,\sigma_t(\xtstar))}
\notag\\
&\hspace{0.5cm}
+ \delta\alpha\norm{\nabla\phit(\xt)\bigg(\yt - \1\sum_{i=1}^{N}\nabla_2\fit(\xit,\sit)\bigg)}
\notag\\
&\stackrel{(a)} 
= 2\delta\norm{\xt -\xtstar}
+ \delta\alpha\norm{G_t(\xt,\st) - \nabla\ft(\xtstar,\sigma_t(\xtstar))}
\notag\\
&\hspace{0.5cm}
+ \delta\alpha L_3\norm{\yt - \1\ytavg},\label{eq:before_dkstar}
\end{align}
where  in \emph{(a)} we use~\eqref{eq:y_mean} and $\norm{\nabla\phit(x)} \le L_3$. The definition of $G_t$ and its Lipschitz continuity (cf. Assumption~\ref{ass:lipschitz}) imply
\begin{align}
&\norm{G_t(\xt,\st) - \nabla\ft(\xtstar,\sigma_t(\xtstar))} 
\notag\\
&\leq L_1\norm{\xt - \xtstar} + L_1 \norm{\st- \1\sigma_t(\xtstar)}.\label{eq:dkavg_gkstar}
\end{align}
By combining~\eqref{eq:before_dkstar} with~\eqref{eq:dkavg_gkstar}, we get
\begin{align}
&\norm{\xtp  -  \xt}  \leq  \delta(2 +\alpha L_1) \norm{\xt -\xtstar} %
+ \delta\alpha L_1\norm{\st-\1\sigma_t(\xtstar)}
\notag\\
&\hspace{0.5cm}
+\delta\alpha L_3\norm{\yt - \1\ytavg} \notag
\\
&\stackrel{(a)} \leq  \delta(2 +\alpha L_1) \norm{\xt -\xtstar} %
+ \delta\alpha L_1\norm{\st-\1\stavg}\notag\\
&\hspace{0.5cm}+\delta\alpha L_1\norm{\1\stavg - \1\sigma_t(\xtstar)}+\delta\alpha L_3\norm{\yt - \1\ytavg},\label{eq:before_sigmak_sigmastar}
\end{align}
where in \emph{(a)} we add and subtract $\1\stavg$ and we apply the triangle inequality.
Now, consider the term $\norm{\1\stavg - \1\sigma_t(\xtstar)}$. By using
the definition of $\sigma_t$ and the Lipschitz continuity of $\phiit$ (cf. Assumption~\ref{ass:lipschitz}),
it can be seen that (see also~\cite{li2020distributed}) ,
\begin{align}
\norm{\1\stavg - \1\sigma_t(\xtstar)} &\leq L_3\norm{\xt - \xtstar}.\label{eq:sigmak_sigmastar}
\end{align}
By combining the results~\eqref{eq:before_sigmak_sigmastar} and~\eqref{eq:sigmak_sigmastar}, the proof is given. \oprocend

\subsection{Proof of Lemma~\ref{lemma:sk_skavg}}
\label{sec:proof_sk_skavg}

By applying~\eqref{eq:sigma_global_update} and~\eqref{eq:sigma_mean_update}, we can write
\begin{align}
&\norm{\stp - \1\stpavg} =\norm{A\st - \1\stavg + \tilde{I}(\phitp(\xtp) - \phit(\xt))}\notag
\\
&\stackrel{(a)}\leq\norm{\left(A - \I\right)\left(\st - \1\stavg\right)}\notag + \norm{\tilde{I}(\phitp(\xtp) - \phit(\xt))},\notag
\end{align}
where \emph{(a)} applies the triangle inequality, introduces $\tilde{I} := I - \I$, and uses the fact that $\1 \in \ker\left(A-\I\right)$. Now, we add and subtract within the second norm the term $\phitp(\xt)$ and we apply the triangle inequality obtaining
\begin{align}
&\norm{\stp - \1\stpavg} \leq\norm{\left(A - \I\right)\left(\st - \1\stavg\right)}\notag
\\
&\hspace{0.3cm}+ \norm{\tilde{I}(\phitp(\xtp) - \phitp(\xt))} + \norm{\tilde{I}(\phitp(\xt) - \phit(\xt))}\notag
\\
&\stackrel{(a)}\leq \Lambda\norm{\st - \1\stavg} + L_3\norm{\xtp - \xt} + \omega_t,\notag
\end{align}
where in \emph{(a)} we use the maximum eigenvalue $\Lambda$ of the matrix $A-\I$, Assumption~\ref{ass:lipschitz}, $\omega_t$ (cf~\eqref{eq:omega}), and $\norm{\tilde{I}} = 1$. By using Lemma~\ref{lemma:xkp_xk} to bound $\norm{\xtp - \xt}$, the proof follows.
\oprocend

\subsection{Proof of Lemma~\ref{lemma:yk_ykavg}}
\label{sec:proof_yk_ykavg}	
We use~\eqref{eq:y_global_update} and~\eqref{eq:y_mean_update} to write
\begin{align}
&\norm{\ytp - \1\ytpavg}\notag \leq\norm{A\yt - \ytavg}\notag
\\
&\hspace{0.5cm}
+ \norm{\left(I-\I\right)(\nabla_2\ftp(\xtp,\stp) - \nabla_2\ftp(\xt,\st)) }\notag
\\
&\stackrel{(a)}\leq\norm{\left(A-\I\right)(\yt-\1\ytavg)}\notag\\
&\hspace{0.5cm}+ \norm{\left(I- \I \right)(\nabla_2\ftp(\xtp,\stp) - \nabla_2\ftp(\xt,\st))}\notag\\
&\hspace{0.5cm}+ \norm{\left(I- \I \right)(\nabla_2\ftp(\xt,\st) - \nabla_2\ft(\xt,\st))},\notag
\end{align}
where \emph{(a)} uses $\1 \in \ker(A-\I)$ and applies the triangle inequality after adding and subtracting $(I- \I)\nabla_2\ftp(\xt,\st)$ within the norm. By using the maximum eigenvalue $\Lambda$ of $A-\I$, Assumption~\ref{ass:lipschitz}, and $\eta_t$ (cf.~\eqref{eq:eta}), we get
\begin{align}
&\norm{\ytp - \1\ytpavg} 
\notag\\
&\leq \Lambda\norm{\yt-\1\ytavg} + L_2\norm{\xtp - \xt} \!+\!  L_2\norm{\stp - \st}  \!+\! \eta_t,\notag
\end{align}
Now, we can use~\eqref{eq:y_global_update} to get
\begin{align}
&\norm{\ytp - \1\ytpavg} \leq
\Lambda\norm{\yt-\1\ytavg} + L_2\norm{\xtp - \xt}
\notag\\
& \hspace{0.5cm}
+ L_2\norm{(A-I)\st + \phitp(\xtp) - \phit(\xt)}+ \eta_t
\notag\\
&\stackrel{(a)}
\leq \Lambda\norm{\yt-\1\ytavg} + L_2\norm{\xtp - \xt} + L_2\norm{(A-I)(\st-\1\stavg)}
\notag\\
&\hspace{0.5cm}
+ L_2\norm{\phitp(\xtp) - \phit(\xt)}  + \eta_t,\notag
\end{align}
where in \emph{(a)} we apply the triangle inequality and the fact that $\1 \in \ker\left(A-\I\right)$.
	We add and subtract within the norm the term $\phitp(\xt)$ and apply the triangle inequality, obtaining
\begin{align}
&\norm{\ytp - \1\ytpavg}\notag \leq \Lambda\norm{\yt-\1\ytavg} + L_2\norm{\xtp - \xt} 
\notag\\
&\hspace{0.5cm}
+ L_2\norm{(A-I)(\st-\1\stavg)}
+ L_2\norm{\phitp(\xtp) - \phitp(\xt)}\notag
\\
&\hspace{0.5cm} + L_2\norm{\phitp(\xt) - \phit(\xt)} + \eta_t\notag
\\
&\stackrel{(a)}\leq\rho\norm{\yt-\1\ytavg} + L_2\norm{\xtp - \xt} + L_2\norm{(A-I)(\st-\1\stavg)}\notag
\\
&\hspace{0.5cm}+ L_2L_3\norm{\xtp - \xt} + L_2\omega_t + \eta_t,\notag
\end{align}
where \emph{(a)} uses Assumption~\ref{ass:lipschitz} and $\omega_t$ (cf.~\eqref{eq:omega}). The proof follows by $\norm{A - I} \leq 2$, and by applying Lemma~\ref{lemma:xkp_xk}.%
\oprocend

\subsection{Proof of Theorem~\ref{th:dyn_regret}}
\label{sec:proof_theorem}

Let us introduce $u_t$ to denote $u_t := \col(\zeta_t,
\eta_t,
\omega_t)$.
Then, by combining Lemma~\ref{lemma:xk_xstar},~\ref{lemma:sk_skavg}, and~\ref{lemma:yk_ykavg}, we bound the evolution of $z_t$ (defined in~\eqref{eq:zt}) through the following dynamical system
\begin{equation}\label{eq:z}
	z_{t+1} \leq M(\delta)z_t + Bu_t,
\end{equation}
in which 
\begin{equation*}
	M(\delta) := M_0 + \delta E, \quad B := \begin{bmatrix}
	1& 0& 0\\
	0& 1& 1\\
	0& 1& L_2
	\end{bmatrix},
\end{equation*}
where
\begin{align*}
	M_0 &:= \begin{bmatrix}
	1&  0& 0\\
	0& \Lambda& 0\\
	0& 2L_2& \Lambda
	\end{bmatrix}, \quad
	E := \begin{bmatrix}
	-\mu\alpha& \alpha L_1& \alpha L_3\\
	E_{21}& \alpha L_1L_3& \alpha L_3^2\\
	E_{31}&  E_{32}& E_{33}\\
	\end{bmatrix},
\end{align*}
with $E_{21} := 2 L_3 + \alpha L_1L_3 + \alpha L_1L_3^2$, $E_{31} := (2 + \alpha L_1 + \alpha L_1)(L_2 + L_2L_3)$, $E_{32} := \alpha L_1(L_2 + L_2L_3)$, and $E_{33} := \alpha L_3(L_2 + L_2L_3)$.
Being $M_0$ triangular, its spectral radius is $1$ 
since $\Lambda \in (0, 1)$ as implied by Assumption~\ref{ass:network}. Denote by $\chi(\delta)$ the eigenvalues of $M(\delta)$ as a
function of $\delta$. Call $v$ and $w$ respectively the right and left eigenvectors of $A_0$ 
associated to $1$. Then, $v = \begin{bmatrix}1&0&0\end{bmatrix}^\top$, $w = \begin{bmatrix}1&0&0\end{bmatrix}^\top$. Being $1$ a simple eigenvalue of $M(0)$, from~\cite[Theorem~6.3.12]{horn2012matrix} it holds
\begin{equation}\nonumber
	\frac{d\chi(\delta)}{d\delta}\bigg|_{\chi=1,\delta=0} = \frac{w^\top Ev}{w^\top v} = -\mu\alpha < 0.
\end{equation}
Then, by continuity of eigenvalues with respect to the matrix entries, there exists $\bar{\delta}>0$ so that $\rhomax(M(\delta)) < 1$ for any $\delta \in (0,\bar{\delta})$. From now on we will omit the dependency of $M$ and its eigenvalues from $\delta$. Since $z_t\geq 0$ for all $t$, and $M$ and $Bu_t$ have only non-negative entries, one can use~\eqref{eq:z} to write
\begin{equation}\label{eq:y_evolution}
	z_{t} \leq M^{t} z_0 + \sum_{k=0}^{t-1} M^{k}Bu_k.
\end{equation}
Pick $\theta \in (0, 1 - \rhomax(M))$ and define $\tilde{\rho} \triangleq \rhomax(M) + \theta$. Then, by~\cite[Lemma~5.6.10]{horn2012matrix},
there exists a matrix norm\footnote{An expression of $\|\cdot\|_\ped$ can be found in the proof of~\cite[Lemma~5.6.10]{horn2012matrix}.}, which we denote as $\|\cdot\|_\ped$, such that
$\| M \|_\ped \leq \rhomax(M) + \theta < 1$. Moreover, by applying~\cite[Theorem~5.7.13]{horn2012matrix}, there exists a vector norm, which we denote by $\| \cdot\|_\ped$,
which is compatible with the corresponding matrix norm, i.e.,
such that $\|Mv\|_\ped \leq \| M \|_\ped\|v\|_\ped$ for any matrix $M \in \R^{3 \times 3}$ and $v \in \R^3$. Using this fact, we use the norm $\|\cdot\|_\ped$ on both sides of~\eqref{eq:y_evolution} and we apply the triangle inequality to get
\begin{align}\notag
\norm{z_{t}}_\ped 
& \leq 
\norm{M^{t}z_0}_\ped + \norm{\sum_{k=0}^{t-1} M^{k}Bu_{t-k-1}}_\ped
\\
& \leq \tilde{\rho}^{t}\norm{z_0}_\ped + \sum_{k=0}^{t-1}\tilde{\rho}^{k}\norm{Bu_{t-k-1}}_\ped.
\label{eq:bound_y}
\end{align}
Being $\nabla\ft$ Lipschitz continuous (cf. Assumption~\ref{ass:lipschitz}), it holds
\begin{align}\label{eq:f_minus_fstar}
\ft(\xt,\sigma_t(\xt)) - \ft(\xtstar,\sigma_t(\xtstar)) &\leq \frac{L_1}{2}\|\xt - \xtstar\|^2
\notag\\
&\stackrel{(a)}{\leq}
\frac{L_1}{2}\|z_t\|^2,
\end{align}
where in \emph{(a)} uses the fact that $\norm{\xt - \xtstar}$ is a component of $z_t$.
Recalling that all norms are equivalent on finite-dimensional vector spaces, there always exist $\lambda_1 > 0$ and $\lambda_2 > 0$ such that $\norm{\cdot} \leq \lambda_1\norm{\cdot}_\ped$ and $\norm{\cdot}_\ped \leq \lambda_2\norm{\cdot}$.
Thus, by exploiting the square norm and combining the results~\eqref{eq:bound_y} with the equivalence of the norms, we can bound~\eqref{eq:f_minus_fstar} as
\begin{align*}
	&\ft(\xt,\sigma_t(\xt)) - \ft(\xtstar,\sigma_t(\xtstar)) 
	\leq 
	\frac{L_1\lambda_1^2}{2} \bigg(\tilde{\rho}^{2t} \norm{z_0}_\ped^2 
	\notag\\
	&
	+ 2\tilde{\rho}^t\norm{z_0}_\ped\sum_{k=0}^{t-1}\tilde{\rho}^{k}\norm{Bu_{t-k-1}}_\ped 
	+ \sum_{k=0}^{t-1}\tilde{\rho}^{2k}\norm{Bu_{t-k-1}}_\ped^2\bigg),
\end{align*}
which, combined with the definitions of $R_T$ (cf.~\eqref{eq:regret}), $U_T$, and $Q_T$ (cf.~\eqref{eq:U_T_Q_T}), and the equivalence of the norms, leads to
\begin{align}
	R_T &\leq \frac{L_1\lambda^2}{2}\left(\sum_{t=1}^T\tilde{\rho}^{2t}\norm{z_0}  + 2\norm{z_0} U_T + Q_T\right)
	\notag\end{align}
	\begin{align}
	&\stackrel{(a)}{\leq}
	\frac{L_1\lambda^2}{2}\left(\frac{\norm{z_0}^2}{1 - \tilde{\rho}^2}  + 2\norm{z_0} U_T + Q_T\right),
\end{align}
where $\lambda = \lambda_1 \lambda_2$ and \emph{(a)} uses the geometric series property.

As regards the result~\eqref{eq:constraint_violation}, we use~\eqref{eq:after_projection} to write 
\begin{align}
	\dist(\xtp,\Xtp)
	&= \dist(\xt + \delta(\txt - \xt),\Xtp)
	\notag\\
	&\stackrel{(a)}{\leq}  \dist(\xt + \delta(\txt - \xt),\Xt) + \gamma_t,\label{eq:constraint_violation_initial}
\end{align}
 where in \emph{(a)} we add and subtract the term $\dist(\xt + \delta(\txt - \xt),\Xt)$ and we introduce $\gamma_t$ (cf.~\eqref{eq:gamma}). Now, we recall that 
 \begin{align*}
	\dist(\xt + \delta(\txt - \xt), \Xt) = \min_{y \in \Xt}\norm{\xt + \delta(\txt - \xt) - y}.
 \end{align*}
 Thus, by adding and subtracting within the norm the term $(1-\delta)v_t$ with $v_t \in \Xt$ so that $\norm{\xt - v_t} = \dist(\xt,\Xt)$, we can use the triangle inequality and the definition of $\min$ to get 
 \begin{align*}
	&\dist(\xt + \delta(\txt - \xt), \Xt) 
	\notag\\
	&\leq (1-\delta)\norm{\xt - v_t}
	+ \min_{y \in \Xt}\norm{v_t + \delta(\txt - v_t) - y}
	\notag\\
	&=
	(1-\delta)\dist(\xt,\Xt) 
	+ \dist(v_t + \delta(\txt - v_t),\Xt),
 \end{align*}
which allows us to rewrite~\eqref{eq:constraint_violation_initial} as
\begin{align}
	\dist(\xtp,\Xtp) &\leq (1-\delta)\dist(\xt,\Xt)
	\notag\\
	&\hspace{0.5cm}
	+ \dist(v_t + \delta(\txt - v_t),\Xt) + \gamma_t.\label{eq:constraint_violation_intermediate}
\end{align}
Notice that $v_t, \txt \in \Xt$ and $0 < \delta < 1$, then $v_t + \delta(\txt - v_t) \in \Xt$ and the second term of~\eqref{eq:constraint_violation_intermediate} is null and~\eqref{eq:constraint_violation_intermediate} becomes
\begin{align}
 	\dist(\xtp,\Xtp) \leq (1-\delta)\dist(\xt,\Xt) + \gamma_t.\label{eq:constraint_violation_intermediate_2}
 \end{align}
Both members of~\eqref{eq:constraint_violation_intermediate_2} are always positive, then~\eqref{eq:constraint_violation_intermediate_2} leads to
\begin{align*}
	\dist(\xt,\Xt) \leq (1 - \delta)^t\dist(x_0,X_0) + \sum_{k=0}^{t-1}(1-\delta)^k \gamma_{t-k-1}.%
\end{align*}
By summing the latter for $t=1$ up to $t=T$ and using the geometric series property the proof follows.\oprocend %

\subsection{Proof of Corollary~\ref{cor:static}}
\label{sec:proof_cor_stat}
	Here, for all $t \ge 0$, it holds $u_t \equiv 0$ and $\xtstar = x^\star$. By the same arguments of Theorem~\ref{th:dyn_regret}, we use the Lipschitz continuity of $\nabla\ft$ (cf. Assumption~\ref{ass:lipschitz}), \eqref{eq:bound_y} with $u_t \equiv 0$, and the equivalence of the norms to get $f(\xt,\sigma(\xt)) - f(x^\star,\sigma(x^\star)) \leq \tilde{\rho}^{2t}\frac{L_1\lambda_1^2}{2}\norm{z^0}_\ped^2$. The proof follows by using again the equivalence of the norms and setting $\lambda := \lambda_1\lambda_2$.
	\oprocend	
\end{appendix}

\bibliographystyle{IEEEtran}

\begin{thebibliography}{10}
	\providecommand{\url}[1]{#1}
	\csname url@samestyle\endcsname
	\providecommand{\newblock}{\relax}
	\providecommand{\bibinfo}[2]{#2}
	\providecommand{\BIBentrySTDinterwordspacing}{\spaceskip=0pt\relax}
	\providecommand{\BIBentryALTinterwordstretchfactor}{4}
	\providecommand{\BIBentryALTinterwordspacing}{\spaceskip=\fontdimen2\font plus
	\BIBentryALTinterwordstretchfactor\fontdimen3\font minus
		\fontdimen4\font\relax}
	\providecommand{\BIBforeignlanguage}[2]{{%
	\expandafter\ifx\csname l@#1\endcsname\relax
	\typeout{** WARNING: IEEEtran.bst: No hyphenation pattern has been}%
	\typeout{** loaded for the language `#1'. Using the pattern for}%
	\typeout{** the default language instead.}%
	\else
	\language=\csname l@#1\endcsname
	\fi
	#2}}
	\providecommand{\BIBdecl}{\relax}
	\BIBdecl
	
	\bibitem{li2020distributed}
	X.~Li, L.~Xie, and Y.~Hong, ``Distributed aggregative optimization over
		multi-agent networks,'' \emph{IEEE Transactions on Automatic Control}, 2021.
	
	\bibitem{koshal2016distributed}
	J.~Koshal, A.~Nedi{\'c}, and U.~V. Shanbhag, ``Distributed algorithms for
		aggregative games on graphs,'' \emph{Operations Research}, vol.~64, no.~3,
		pp. 680--704, 2016.
	
	\bibitem{liang2017distributed}
	S.~Liang, P.~Yi, and Y.~Hong, ``Distributed nash equilibrium seeking for
		aggregative games with coupled constraints,'' \emph{Automatica}, vol.~85, pp.
		179--185, 2017.
	
	\bibitem{gadjov2018passivity}
	D.~Gadjov and L.~Pavel, ``A passivity-based approach to nash equilibrium
		seeking over networks,'' \emph{IEEE Transactions on Automatic Control},
		vol.~64, no.~3, pp. 1077--1092, 2018.
	
	\bibitem{yi2019operator}
	P.~Yi and L.~Pavel, ``An operator splitting approach for distributed
		generalized nash equilibria computation,'' \emph{Automatica}, vol. 102, pp.
		111--121, 2019.
	
	\bibitem{belgioioso2020distributed}
	G.~Belgioioso, A.~Nedi{\'c}, and S.~Grammatico, ``Distributed generalized nash
		equilibrium seeking in aggregative games on time-varying networks,''
		\emph{IEEE Transactions on Automatic Control}, vol.~66, no.~5, pp.
		2061--2075, 2020.
	
	\bibitem{notarstefano2019distributed}
	G.~Notarstefano, I.~Notarnicola, and A.~Camisa, ``Distributed optimization for
		smart cyber-physical networks,'' \emph{Foundations and
		Trends{\textregistered} in Systems and Control}, vol.~7, no.~3, pp. 253--383,
		2019.
	
	\bibitem{cavalcante2013distributed}
	R.~L. Cavalcante and S.~Stanczak, ``A distributed subgradient method for
		dynamic convex optimization problems under noisy information exchange,''
		\emph{IEEE Journal of Selected Topics in Signal Processing}, vol.~7, no.~2,
		pp. 243--256, 2013.
	
	\bibitem{towfic2014adaptive}
	Z.~J. Towfic and A.~H. Sayed, ``Adaptive penalty-based distributed stochastic
		convex optimization,'' \emph{IEEE Transactions on Signal Processing},
		vol.~62, no.~15, pp. 3924--3938, 2014.
	
	\bibitem{akbari2015distributed}
	M.~Akbari, B.~Gharesifard, and T.~Linder, ``Distributed online convex
		optimization on time-varying directed graphs,'' \emph{IEEE Transactions on
		Control of Network Systems}, vol.~4, no.~3, pp. 417--428, 2015.
	
	\bibitem{mateos2014distributed}
	D.~Mateos-N{\'u}nez and J.~Cort{\'e}s, ``Distributed online convex optimization
		over jointly connected digraphs,'' \emph{IEEE Transactions on Network Science
		and Engineering}, vol.~1, no.~1, pp. 23--37, 2014.
	
	\bibitem{hosseini2016online}
	S.~Hosseini, A.~Chapman, and M.~Mesbahi, ``Online distributed convex
		optimization on dynamic networks,'' \emph{IEEE Transactions on Automatic
		Control}, vol.~61, no.~11, pp. 3545--3550, 2016.
	
	\bibitem{yuan2017adaptive}
	D.~Yuan, D.~W. Ho, and G.-P. Jiang, ``An adaptive primal-dual subgradient
		algorithm for online distributed constrained optimization,'' \emph{IEEE
		transactions on cybernetics}, vol.~48, no.~11, pp. 3045--3055, 2017.
	
	\bibitem{shahrampour2017distributed}
	S.~Shahrampour and A.~Jadbabaie, ``Distributed online optimization in dynamic
		environments using mirror descent,'' \emph{IEEE Transactions on Automatic
		Control}, vol.~63, no.~3, pp. 714--725, 2017.
	
	\bibitem{zhou2017incentive}
	X.~Zhou, E.~Dall'Anese, L.~Chen, and A.~Simonetto, ``An incentive-based online
		optimization framework for distribution grids,'' \emph{IEEE transactions on
		Automatic Control}, vol.~63, no.~7, pp. 2019--2031, 2017.
	
	\bibitem{akbari2019individual}
	M.~Akbari, B.~Gharesifard, and T.~Linder, ``Individual regret bounds for the
		distributed online alternating direction method of multipliers,'' \emph{IEEE
		Transactions on Automatic Control}, vol.~64, no.~4, pp. 1746--1752, 2019.
	
	\bibitem{lee2017sublinear}
	S.~Lee and M.~M. Zavlanos, ``On the sublinear regret of distributed primal-dual
		algorithms for online constrained optimization,'' \emph{arXiv preprint
		arXiv:1705.11128}, 2017.
	
	\bibitem{li2020distributedConstraints}
	X.~Li, X.~Yi, and L.~Xie, ``Distributed online optimization for multi-agent
		networks with coupled inequality constraints,'' \emph{IEEE Transactions on
		Automatic Control}, 2020.
	
	\bibitem{yi2020distributed}
	X.~Yi, X.~Li, L.~Xie, and K.~H. Johansson, ``Distributed online convex
		optimization with time-varying coupled inequality constraints,'' \emph{IEEE
		Transactions on Signal Processing}, vol.~68, pp. 731--746, 2020.
	
	\bibitem{li2020distributedOnline}
	X.~Li, X.~Yi, and L.~Xie, ``Distributed online convex optimization with an
		aggregative variable,'' \emph{IEEE Transactions on Control of Network
		Systems}, 2021.
	
	\bibitem{zhu2010discrete}
	M.~Zhu and S.~Mart{\'\i}nez, ``Discrete-time dynamic average consensus,''
		\emph{Automatica}, vol.~46, no.~2, pp. 322--329, 2010.
	
	\bibitem{kia2019tutorial}
	S.~S. Kia, B.~Van~Scoy, J.~Cortes, R.~A. Freeman, K.~M. Lynch, and S.~Martinez,
		``Tutorial on dynamic average consensus: The problem, its applications, and
		the algorithms,'' \emph{IEEE Control Systems Magazine}, vol.~39, no.~3, pp.
		40--72, 2019.
	
	\bibitem{shi2015extra}
	W.~Shi, Q.~Ling, G.~Wu, and W.~Yin, ``{EXTRA}: An exact first-order algorithm
		for decentralized consensus optimization,'' \emph{SIAM Journal on
		Optimization}, vol.~25, no.~2, pp. 944--966, 2015.
	
	\bibitem{varagnolo2015newton}
	D.~Varagnolo, F.~Zanella, A.~Cenedese, G.~Pillonetto, and L.~Schenato,
		``{N}ewton-{R}aphson consensus for distributed convex optimization,''
		\emph{IEEE Transactions on Automatic Control}, vol.~61, no.~4, pp. 994--1009,
		2015.
	
	\bibitem{dilorenzo2016next}
	P.~Di~Lorenzo and G.~Scutari, ``{NEXT}: In-network nonconvex optimization,''
		\emph{IEEE Transactions on Signal and Information Processing over Networks},
		vol.~2, no.~2, pp. 120--136, 2016.
	
	\bibitem{nedic2017achieving}
	A.~Nedi{\'c}, A.~Olshevsky, and W.~Shi, ``Achieving geometric convergence for
		distributed optimization over time-varying graphs,'' \emph{SIAM Journal on
		Optimization}, vol.~27, no.~4, pp. 2597--2633, 2017.
	
	\bibitem{qu_harnessing_2018}
	G.~Qu and N.~Li, ``Harnessing {Smoothness} to {Accelerate} {Distributed}
		{Optimization},'' \emph{IEEE Transactions on Control of Network Systems},
		vol.~5, no.~3, pp. 1245--1260, 2018.
	
	\bibitem{xu2017convergence}
	J.~Xu, S.~Zhu, Y.~C. Soh, and L.~Xie, ``Convergence of asynchronous distributed
		gradient methods over stochastic networks,'' \emph{IEEE Transactions on
		Automatic Control}, vol.~63, no.~2, pp. 434--448, 2017.
	
	\bibitem{xi2017add}
	C.~Xi, R.~Xin, and U.~A. Khan, ``{ADD-OPT}: Accelerated distributed directed
		optimization,'' \emph{IEEE Transactions on Automatic Control}, vol.~63,
		no.~5, pp. 1329--1339, 2017.
	
	\bibitem{xin2018linear}
	R.~Xin and U.~A. Khan, ``A linear algorithm for optimization over directed
		graphs with geometric convergence,'' \emph{IEEE Control Systems Letters},
		vol.~2, no.~3, pp. 315--320, 2018.
	
	\bibitem{scutari2019distributed}
	G.~Scutari and Y.~Sun, ``Distributed nonconvex constrained optimization over
		time-varying digraphs,'' \emph{Mathematical Programming}, vol. 176, no. 1-2,
		pp. 497--544, 2019.
	
	\bibitem{zhang2019distributed}
	Y.~Zhang, R.~J. Ravier, M.~M. Zavlanos, and V.~Tarokh, ``A distributed online
		convex optimization algorithm with improved dynamic regret,'' in \emph{{IEEE}
		Conference on Decision and Control {(CDC)}}, 2019, pp. 2449--2454.
	
	\bibitem{notarnicola2020personalized}
	I.~Notarnicola, A.~Simonetto, F.~Farina, and G.~Notarstefano, ``Distributed
		personalized gradient tracking with convex parametric models,'' \emph{IEEE
		Transactions on Automatic Control}, 2022.
	
	\bibitem{carnevale2020distributed}
	G.~Carnevale, F.~Farina, I.~Notarnicola, and G.~Notarstefano, ``Distributed
		online optimization via gradient tracking with adaptive momentum,''
		\emph{arXiv preprint arXiv:2009.01745}, 2020.
	
	\bibitem{bertsekas2015convex}
	D.~P. Bertsekas and A.~Scientific, \emph{Convex optimization algorithms}.\hskip
		1em plus 0.5em minus 0.4em\relax Athena Scientific Belmont, 2015.
	
	\bibitem{horn2012matrix}
	R.~A. Horn and C.~R. Johnson, \emph{Matrix analysis}.\hskip 1em plus 0.5em
		minus 0.4em\relax Cambridge university press, 2012.
	
	\end{thebibliography}

\end{document}